\documentclass[10pt]{amsart}

\usepackage{amsaddr}
\usepackage{amscd,amssymb}

\title[Epsilon-strong systems]{Simplicity of algebras via epsilon-strong systems}

\newtheorem{thm}{Theorem}
\newtheorem{prop}[thm]{Proposition}
\newtheorem{lem}[thm]{Lemma}
\newtheorem{cor}[thm]{Corollary}
\newtheorem{que}[thm]{Question}
\theoremstyle{definition}
\newtheorem{defi}[thm]{Definition}
\newtheorem{exa}[thm]{Example}
\newtheorem{rem}[thm]{Remark}

\begin{document}

\author{Patrik Nystedt}
\address{University West,
Department of Engineering Science, 
SE-46186 Trollh\"{a}ttan, Sweden}

\email{patrik.nystedt@hv.se}

\subjclass[2010]{16S99, 16W22, 16W55, 22A22, 37B05}

\keywords{graded ring, skew inverse semigroup ring, Steinberg algebra, partial action}

\begin{abstract}
We obtain sufficient criteria for simplicity of 
systems, that is, rings $R$ that are equipped
with a family of additive subgroups $R_s$, for $s \in S$,
where $S$ is a semigroup, satisfying $R = \sum_{s \in S} R_s$ and
$R_s R_t \subseteq R_{st}$, for $s,t \in S$.
These criteria are specialized to obtain sufficient criteria for
simplicity of, what we call, s-unital epsilon-strong systems, that is
systems where $S$ is an inverse semigroup, $R$ is coherent,
in the sense that for all $s,t \in S$ with $s \leq t$,
the inclusion $R_s \subseteq R_t$ holds, and for each $s \in S$,
the $R_s R_{s^*}$-$R_{s^*}R_s$-bimodule $R_s$ is s-unital.
As an aplication of this, we obtain generalizations of recent 
criteria for simplicity of skew inverse semigroup rings, 
by Beuter, Goncalves, \"{O}inert and Royer, and then, in turn, 
for Steinberg algebras, over non-commutative rings, 
by Brown, Farthing, Sims, Steinberg, Clark and Edie-Michel. 
\end{abstract}

\maketitle

\section{Introduction}\label{introduction}

Steinberg algebras were independently introduced in \cite{clark2014} and 
\cite{steinberg2010} and they are closely related to the constructions in \cite{exel2008}.
Steinberg algebras are algebrai\-sations of Renault's $C^*$-algebras of groupoids
and can be viewed as a model for discrete inverse semigroup algebras. 
Lately, Steinberg algebras have attracted a lot of attention due to the fact that
they include many well-known constructions, such as, the Kumjian-Pask algebras 
of \cite{arandopino2013} and \cite{clarkflynn2014} which, in turn, include
the Leavitt path algebras of \cite{tomforde2011}.
Many different properties of Steinberg algebras have been studied,
such as when they are primitive, semiprimitive, artinian or noetherian
(see e.g. \cite{steinberg2016} and \cite{steinberg2018}).
In this article, we focus on the property of simplicity of Steinberg algebras. 
Necessary and sufficient conditions for this was first obtained for complex algebras
in \cite{brown2014}. This result was generalized to algebras over general
fields in \cite{steinberg2016} and in \cite{clark2015} to algebras over 
commutative unital rings:

\begin{thm}[Clark and Edie-Michel \cite{clark2015}]\label{theoremsteinberg}
If $G$ is a Hausdorff and ample groupoid, and $K$ is a commutative unital ring,
then the Steinberg algebra $A_K(G)$ is simple if and only if
$G$ is effective and minimal, and $K$ is a field.
\end{thm}

In \cite{clark2018} Clark, Exel, Pardo, Sims and Starling have found
necessary and sufficient criteria for simplicity of the Steinberg 
algebra $A_K(G)$ when $K$ is a field and $G$ is a second countable, not necessarily
Hausdorff, groupoid. In this article, we will, however, only consider Hausdorff groupoids.

The class of partial skew group rings were introduced by Dokuchaev and Exel
in \cite{dokuchaev2005} as a generalization of classical skew group rings and
as an algebraic analogue of partial crossed product C*-algebras.
This class of rings have been studied a lot mainly since 
many other constructions, such as the Leavitt path algebras \cite{goncalves2014}
and ultragraph Leavitt path algebras \cite{goncalves2017}, 
can be realized as partial skew group rings.
The class of partial skew group rings have been generalized even 
further with the definition of skew inverse semigroup rings by 
Exel and Vieira in \cite{exel2010}.
From a result of Beuter and Goncalves in \cite{beuter2018} it follows 
that every Steinberg algebra can be seen as a skew inverse semigroup ring. 
This means that results concerning skew inverse semigroup 
rings can be translated to results concerning Steinberg algebras.
Such a translation was recently made by Beuter, Goncalves, \"{O}inert and Royer
in \cite{beuter2017} where they deduce Theorem \ref{theoremsteinberg} from
the following result.

\begin{thm}[Beuter, Goncalves, \"{O}inert and Royer \cite{beuter2017}]\label{thmbeuter}
If $\pi$ is a locally unital partial action of an inverse semigroup $S$ 
on an associative, commutative and locally unital ring $A$,
then the skew inverse semigroup ring $A \rtimes_{\pi} S$ is simple 
if and only if $A$ is $S$-simple and $A$ is a maximal commutative subring of $A \rtimes_{\pi} S$.
\end{thm}

The motivation for the present article is derived from the following remark
made by Beuter and Goncalves in \cite[Definition 2.9]{beuter2018}.
If $R = A \rtimes_{\pi} S$ is a partial skew inverse semigroup ring
and we for all $s \in S$ put $R_s = \overline{D_s \delta_s}$,
then $R$ is a {\it system} in the sense that
$R = \sum_{s \in S} R_s$ and for all $s,t \in S$ the inclusion $R_s R_t \subseteq R_{st}$ holds,
and $R$ is {\it coherent} in the sense that
for all $s,t \in S$ with $s \leq t$ the inclusion $R_s \subseteq R_t$ holds.
This observation prompts the author of the present article to ask the following: 

\begin{que}
Is there a generalization of Theorem \ref{thmbeuter}, valid for coherent systems? 
\end{que}

In the main result of this article (see Theorem \ref{maintheorem} below),
we partially answer this question.
Namely, we provide sufficient conditions for simplicity of coherent systems.
Before we state this result, let us briefly describe the objects involved in this context.
Suppose that $A$ is a subring of $R$.
The {\it centralizer} of $A$ in $R$, denoted by $C_R(A)$, is the 
set of elements in $R$ that commute with every element of $A$.
If $C_R(A) = A$, then $A$ is said to be a {\it maximal commutative} 
subring of $R$. The set $C_A(A)$ is called the {\it center} of $A$
and is denoted by $Z(A)$.
Let $S$ be an inverse semigroup and let $R$ be a system.
Let $E(S)$ denote the set of all idempotents of $S$ and
put $R_0 = \sum_{e \in E(S)} R_s$.
We say that an ideal $I$ of $R$ is a {\it system ideal} if $I = \sum_{s \in S} I \cap R_s$
and we say that $R$ is {\it system simple} if $R$ and $\{ 0 \}$ are the only
system ideals of $R$. 
We say that $R$ is a left (right) {\it s-unital epsilon-strong system} if for all $s \in S$
the left $R_s R_{s^*}$-module (right $R_{s^*} R_s$-module) $R_s$ is s-unital.
Note that if $S$ is a group and $R$ is $S$-graded, then $R$ is
an epsilon-strong system precisely when it is epsilon-strongly graded
in the sense of Nystedt, \"{O}inert and Pinedo \cite{NOP2016}.

\begin{thm}\label{maintheorem}
If $S$ is an inverse semigroup and 
$R$ is a system simple cohe\-rent left (or right) $s$-unital epsilon-strong system
and $C_R( Z(R_0) ) \subseteq R_0$, then $R$ is simple.
\end{thm}

We use this result to obtain non-commutative and non-unital generali\-zations
of Theorem \ref{theoremsteinberg} and Theorem \ref{thmbeuter}
(see Theorem \ref{secondmaintheorem} and Theorem \ref{thirdmaintheorem} below).
Here is a detailed outline of the article.
In Section \ref{modules}, we recall some definitions and results
concerning unital, locally unital and s-unital modules that we need in the sequel.
In Section \ref{systems}, we recall the relevant background on systems and graded rings. 
We obtain sufficient conditions for simplicity for systems (see Theorem \ref{ZR0MAX}).
In Section \ref{epsilonstrongsystems}, we define left and right 
unital (s-unital) epsilon-strong systems (see Definition \ref{defunitalepsilon})
as well as giving some equivalent conditions characterizing them
(see Proposition \ref{epsilonequivalent} and Corollary \ref{epsiloncor}).
At the end of this section, we prove Theorem \ref{maintheorem}.
In Section \ref{partialskewsection}, we recall the definition of
skew inverse semigroup rings and
we determine when such rings are left (right) epsilon-strong systems
(see Proposition \ref{associative}).
At the end of this section, using Theorem \ref{maintheorem}, 
we prove the following s-unital generalization of Theorem \ref{thmbeuter}.

\begin{thm}\label{secondmaintheorem}
Suppose that $\pi$ is an s-unital partial action of 
an inverse semigroup $S$ on an associative (but not necessarily commutative) 
$s$-unital ring $A$.
If $A$ is $S$-simple and $C_{A \rtimes_{\pi} S}( Z(A) ) \subseteq A$, 
then $A \rtimes_{\pi} S$ is simple.
If $A$ is commutative, then $A \rtimes_{\pi} S$ is simple if and only if
$A$ is $S$-simple and $A$ is a maximal 
commutative subring of $A \rtimes_{\pi} S$.
\end{thm}

In Section \ref{sectionsteinbergalgebras}, we recall from \cite{beuter2017}
the description of Steinberg algebras as partial skew inverse 
semigroup rings. We use this description and Theorem \ref{secondmaintheorem}
to prove the following non-commutative and s-unital 
generalisation of Theorem \ref{theoremsteinberg}.

\begin{thm}\label{thirdmaintheorem}
Suppose that $K$ is a non-zero and associative (but not necessarily commu\-tative
or unital) ring with the property that $Z(K)$ contains a set of s-units for $K$.
If $G$ is a Hausdorff and ample groupoid, then the Steinberg algebra $A_K(G)$
is simple if and only if $G$ is effective and minimal, and $K$ is simple.
\end{thm}

In Section \ref{gradedrings}, we specialize the above results
to groupoid (and group) graded rings.
In particular, we obtain necessary and sufficient conditions
for partial skew groupoid rings, over commutative rings, to
be simple (see Theorem \ref{oinertgen}).

\section{Unital, locally unital and s-unital modules}\label{modules}

In this section, we recall some definitions and results
concerning unital, locally unital and s-unital modules that we need in the sequel.
Throughout this article all rings are supposed to be associative, but,
unless otherwise stated, not necessarily commutative or unital.
For the remainder of this section, let
$A$ and $B$ be rings and let $M$ be an $A$-$B$-bimodule.
If $X \subseteq A$ ($Y \subseteq B$), then we let $XM$ ($MY$) 
denote the set of all finite sums of elements of the 
form $xm$ ($my$), for $x \in X$ ($y \in Y$) and $m \in M$.
We let $\mathbb{N}$ denote the set of positive integers
and we let $\mathbb{Z}_{\geq 0}$ denote the set of non-negative integers.

\begin{defi}
Recall that $M$ is said to be left (right) {\it unital} if there exists
$a \in A$ ($b \in B$) such that for all $m \in M$ the relation $a m = m$ ($mb=m$) holds.
In that case, $a$ ($b$) is said to be a left (right) {\it identity} for $M$.
$M$ is said to be a unital as an $A$-$B$-bimodule,
if it is unital both as a left $A$-module and a right $B$-module.
The ring $A$ is said to be left (right) unital if it is
left (right) unital as a left (right) module over itself.
The ring $A$ is called called unital if it is unital 
as an $A$-bimodule over itself.
\end{defi}

\begin{rem}\label{twosided}
Note that if $A=B$ so that $M$ is a unital $A$-bimodule, then there is $c \in A$
which is simultaneously an left identity and a right identity for $M$.
In fact, if $a \in A$ is a left identity for $M$ and $b \in A$ is a right identity for $M$,
then $c = a + b - ba \in A$ is a two-sided identity for $M$.
\end{rem}

\begin{defi}
If $A=B$, then the $A$-bimodule $M$ is called
{\it locally unital} if for all finite subsets $X$ of $M$ there exists
an idempotent $e$ in $A$ such that $X \subseteq e X e$.
The ring $A$ is called locally unital if it is locally unital 
as a bimodule over itself. 
\end{defi}

\begin{exa}
There are lots of examples of locally unital rings. For instance, if $\{ A_i \}_{i \in I}$
is a family of non-zero unital rings and we let $A$ be the direct sum of the rings $A_i$,
then $A$ is locally unital. However, $A$ is unital if and only if $I$ is finite.
\end{exa}

\begin{defi}
Following H. Tominaga \cite{tominaga1976} we say that $M$
is left (right) $s$-{\it unital} if for all $m \in M$ there exists
$a \in A$ ($b \in B$) such that $am = m$ ($mb = m$).
The $A$-$B$-bimodule $M$ is said to be s-unital if it is 
s-unital both as a left $A$-module and a right $B$-module.
The ring $A$ is said to be left (right) $s$-unital
if it is left (right) $s$-unital as a left (right) module over itself.
The ring $A$ is said to be $s$-unital if it is s-unital
as a bimodule over itself.
\end{defi}

\begin{exa}\label{notlocallyunital}
If we let $X$ be a compact Hausdorff topological space, then
the ring $C_0(X)$ of compactly supported continuous functions $X \to \mathbb{R}$
is s-unital. However, $C_0(X)$ is locally unital if and only if $X$ is compact
(in which case $C_0(X)$ is unital).
\end{exa}

\begin{exa}
The following example (inspired by \cite[Exercise 1.10]{lam2003}) 
shows that there are lots of examples of rings which are left
s-unital but not right s-unital. 
Let $A$ be a unital ring with a non-zero multiplicative identity 1. 
Let $B$ denote the set $A \times A$  equipped with componentwise addition,
and multiplication defined by the relations $(a,b) (c,d) = (ac,ad)$, for $a,b,c,d \in A$. 
It is easy to check that $B$ is associative. 
It is clear that any element of the form $(1,a)$, for $a \in A$, 
is a left identity for $B$. 
However, $B$ is not right unital. 
Indeed, since $(0,1) \notin \{ (0,0) \} = (0,1) B$ 
it follows that $B$ is not even right s-unital. 
For each $n \in \mathbb{N}$ let $C_n$ denote a copy of $B$ and put 
$C = \oplus_{n \in \mathbb{N}} C_n$. 
Then $C$ is left s-unital but not left unital. 
Since none of the $C_n$ are right s-unital it follows that $C$ is not right s-unital. 
\end{exa}

\begin{prop}\label{tominaga}
Let $M$ be an $A$-$B$-bimodule.
\begin{itemize}

\item[(a)] $M$ is left (right)
$s$-unital if and only if for all $n \in \mathbb{N}$ and
all $m_1,\ldots,m_n$ in $M$ there is $a \in A$ ($b \in B$) such that
for all $i \in \{ 1,\ldots,n \}$ the relation $a m_i = m_i$ ($m_i b = m_i$) holds.

\item[(b)] If $A=B$, then the $A$-bimodule $M$ is s-unital
if and only if for all $n \in \mathbb{N}$ and
all $m_1,\ldots,m_n \in M$ there is $c \in A$ such that
for all $i \in \{ 1,\ldots,n \}$ the relations $c m_i = m_i c = m_i$ hold.

\item[(c)] The ring $A$ is s-unital if and only if 
for all $n \in \mathbb{N}$ and all $a_1,\ldots,a_n$ in $A$ 
there is $c \in A$ such that
for all $i \in \{ 1,\ldots,n \}$ the relations $c a_i = a_i c = a_i$ hold.

\end{itemize}
\end{prop}

\begin{proof}
(a) This is \cite[Theorem 1]{tominaga1976}.
(b) Follows if we use the same argument as in the unital case 
in Remark \ref{twosided}.
(c) Follows from (b).
\end{proof}

\begin{defi}\label{unitary}
Following \'{A}nh and M\'{a}rki \cite{anh1987} we say that
$M$ is left (right) {\it unitary} if $AM = M$ ($MB = M$).
\end{defi}

\section{Systems}\label{systems}

In this section, we recall the relevant background on systems and graded rings
and we obtain sufficient conditions for simplicity for systems (see Theorem \ref{ZR0MAX}).
Throughout this section, $S$ denotes a {\it semigroup},
that is a non-empty set equipped
with an associative binary operation 
$S \times S \ni (s,t) \mapsto st \in S$,
and $R$ denotes a {\it system}. Recall, from the introduction, that the latter means that
there to every $s \in S$ is an additive subgroup $R_s$ of $R$ 
such that $R = \sum_{s \in S} R_s$ and for all $s,t \in S$
the inclusion $R_s R_t \subseteq R_{st}$ holds.

\begin{defi}
The ring $R$ is called a {\it strong system} if for all
$s,t \in S$ the equality $R_s R_t = R_{st}$ holds.
The ring $R$ is called {\it graded} if $R = \oplus_{s \in S} R_s$. 
If $R$ is graded, then $R$ is called {\it strongly graded} if it is also a strong system.
Let $E(S)$ denote the set of idempotents of $S$ and
put $R_0 = \sum_{e \in E(S)} R_e$.
We say that $R$ is {\it idempotent coherent} if for all $s \in S$
the inclusions $R_0 R_s \subseteq R_s$ and $R_s R_0 \subseteq R_s$ hold.
In that case, clearly, $R_0$ is a subring of $R$.
\end{defi}

Now we extend a definition from the group (or groupoid) graded case \cite{cohen1983} 
(or \cite{lundstrom2012}) to the semigroup system situation.

\begin{defi}
We say that $R$ is left (right) {\it non-degenerate} if 
for all all $s \in E(S)$ and all non-zero $r \in R_s$, 
there is $t \in S$ such that $t s \in E(S)$ ($s t \in E(S)$) and
$R_t r$ is non-zero ($r R_t$ is non-zero).
\end{defi}

\begin{defi}\label{degreemap}
In the sequel we will use the function 
$d : R \rightarrow \mathbb{Z}_{\geq 0}$ defined in the following way.
If $r = 0$, then put $d(r)=0$.
Now suppose that $r \neq 0$. 
Then there is $n \in \mathbb{N}$,
$s_1,\ldots,s_n \in S$ and non-zero $r_i \in R_{s_i}$,
for $i = 1,\ldots,n$, such that $r = \sum_{i=1}^n r_i$.
Amongst all such representations of $r$, choose one
with $n$ minimal. Put $d(r)=n$.
If $I$ is an ideal and $r \in I$, then we say that $r$ is $I$-{\it minimal}
if $d(r) = {\rm min} \{ d(r') \mid r' \in I \setminus \{ 0 \} \}$.
\end{defi}

\begin{defi}
Suppose that $A/B$ is a {\it ring extension} i.e. that
$A$ and $B$ are rings with $A \supseteq B$.
Recall that the {\it centralizer} of $B$ in $A$, denoted by $C_A(B)$, is the 
set of elements in $A$ that commute with every element of $B$.
If $C_A(B) = B$, then $B$ is said to be a {\it maximal commutative 
subring} of $A$. The set $C_A(A)$ is called the {\it center} of $A$
and is denoted by $Z(A)$.
The ring extension $A/B$ is said to have the {\it ideal 
intersection property} if every non-zero ideal of $A$
has non-zero intersection with $B$.
\end{defi}

\begin{prop}\label{propintersection}
If $R$ is idempotent coherent and left (right) non-degenerate,
then $R / C_R ( Z(R_0) )$ has the ideal intersection property. 
\end{prop}

\begin{proof}
Take a non-zero ideal $I$ of $R$ and an 
$I$-minimal element $r$. Take $n \in \mathbb{N}$ such that $d(r)= n$. 
Choose $s_1,\ldots,s_n \in S$ and non-zero $r_i \in R_{s_i}$,
for $i = 1,\ldots,n$, such that $r = \sum_{i=1}^n r_i$.

Case 1: $R$ is left non-degenerate.
Choose $t \in S$, $i \in \{ 1 , \ldots , n \}$ and
$x \in R_t$ such that $t s_i \in E(S)$
and $x r \neq 0$. Put $r' = x r$. Then $d(r') \leq d(r)$
and thus $r'$ is $I$-minimal.
Take $w \in Z(R_0)$.
Since $r \in I$, we get that 
$wr' - r'w \in I$.
However $wr' - r'w = \sum_{j=1}^n ( w x r_j - x r_j w )=
\sum_{j=1, \ j \neq i}^n ( w x r_j - x r_j w)$
since $x r_i \in R_t R_{s_i} \subseteq R_{t s_i} \subseteq R_0$.
For each $j \in \{ 1,\ldots,n \}$ with $j \neq i$ it holds that
$w x r_j - x r_j w \in R_0 R_t R_{s_j} + R_t R_{s_j} R_0
\subseteq R_{t s_j}$ since $R$ is idempotent coherent.
Thus $d( w r' - r' w ) < n = d(r')$.
From $I$-minimality of $r'$ we conclude that $wr' = r'w$.
Hence $r' \in I \cap C_R ( Z(R_0) )$.

Case 2: $R$ is right non-degenerate.
Similar to the proof of Case 1 and is therefore left to the reader.
\end{proof}

\begin{defi}
Let $I$ be an ideal of $R$.
We say that $I$ is a system ideal if $I = \sum_{s \in S} I \cap R_s$.
We say that $R$ is system simple if $R$ and $\{ 0 \}$ are the only
system ideals of $R$. 
\end{defi}

\begin{rem}\label{simpleremark}
If $R$ is simple, then, clearly, $R$ is system simple.
\end{rem}

\begin{thm}\label{ZR0MAX}
If $R$ is idempotent coherent, system simple, left (or right) 
non\-degene\-rate and $C_R( Z(R_0) ) \subseteq R_0$, then $R$ is simple.
\end{thm}

\begin{proof}
Let $I$ be a non-zero ideal of $R$.
From Proposition \ref{propintersection} it follows that 
the additive group $J = I \cap C_R( Z(R_0) )$ is non-zero.
From the assumption $C_R( Z(R_0) ) \subseteq R_0$ it follows that $J \subseteq R_0$.
Thus $K = R J R + J$ is a non-zero system ideal of $R$.
From system simplicity of $R$ it follows that $K = R$.
Thus $R = K = RJR + J \subseteq RIR + I = I$ and hence $R = I$.
\end{proof}

\begin{cor}\label{corcomm1}
If $R$ is idempotent coherent, left (or right) non-degenerate
and $R_0$ is maximally commutative in $R$, 
then $R$ is simple if and only if $R$ is system simple
\end{cor}

\begin{proof}
This follows from Remark \ref{simpleremark} and Theorem \ref{ZR0MAX}.
\end{proof}

\section{Epsilon-strong systems}\label{epsilonstrongsystems}

At the end of this section, we prove Theorem \ref{maintheorem}
(see Theorem \ref{newmaintheorem}).
We introduce left and right unital (s-unital)
epsilon-strong systems (see Definition \ref{defunitalepsilon})
and obtain some characterizations of these objects
(see Proposition \ref{epsilonequivalent} and Corollary \ref{epsiloncor}).
Throughout this section, $S$ denotes an {\it inverse semigroup}.
Recall that this means that there for all
$s \in S$ exists a unique $t \in S$ such that $sts = s$ and $tst = t$.
We will use the standard notation and put $s^* = t$. 
There is a partial order $\leq$ on $S$ defined by
saying that if $s,t \in S$, then $s \leq t$ if $s = t s^* s$.
For the rest of the section, $R$ denotes a system.
It is easy to see that for all $s \in S$, $R_s R_{s^*}$ is a ring
and $R_s$ is an $R_s R_{s^*}$-$R_{s^*} R_s$-bimodule.

\begin{defi}\label{defunitalepsilon}
Let $\mathcal{P}$ denote either ''unital'' or ''s-unital'' or ''unitary''.
We say that $R$ is {\it left (right)} $\mathcal{P}$ {\it epsilon-strong} if for all $s \in S$
the left $R_s R_{s^*}$-module (right $R_{s^*} R_s$-module) $R_s$ is $\mathcal{P}$.
If $R$ is both left and right $\mathcal{P}$ epsilon-strong,
then we say that $R$ is $\mathcal{P}$ epsilon-strong.
\end{defi}

\begin{rem}
Note that $R$ is left (or right) unitary epsilon-strong if and only if
$R$ is symmetric in the sense of \cite[Definition 4.5]{CEP2016}, that is
if for all $s \in S$, the equality $R_s R_{s^*} R_s = R_s$ holds.
\end{rem}

\begin{prop}\label{epsilonequivalent}
If $\mathcal{P}$ denotes either ''unital'' or ''s-unital'',
then (i), (ii) and (iii) below are equivalent.
\begin{itemize}

\item[(i)] $R$ is left (right) $\mathcal{P}$ epsilon-strong.

\item[(ii)] $R$ is symmetric and for all $s \in S$, the ring
$R_s R_{s^*}$ is left (right) $\mathcal{P}$.

\item[(iii)] 
\begin{itemize}

\item[$\bullet$] $\mathcal{P}$ = unital:
for all $s \in S$, there exists $\epsilon_s \in R_s R_{s^*}$
($\epsilon_s' \in R_{s^*} R_s$) such that for all $r \in R_s$,
$\epsilon_s r = r$ ($r \epsilon_s' = r$).

\item[$\bullet$] $\mathcal{P}$ = s-unital:
for all $s \in S$ and all $r \in R_s$, there exists 
$\epsilon_s \in R_s R_{s^*}$ ($\epsilon_s' \in R_{s^*} R_s$)
with $\epsilon_s r = r$ ($r \epsilon_s' = r$).

\end{itemize}
\end{itemize}
\end{prop}

\begin{proof}
We only show the ''left'' parts of the proof. The ''right'' parts 
are shown in an analogous way and is therefore left to the reader.

$\bullet$ $\mathcal{P}$ = unital:

(i)$\Rightarrow$(ii): 
Take $s \in S$. Since $R_s$ is a unital left
$R_s R_{s^*}$-module it follows immediately that $(R_s R_{s^*})R_s = R_s$.
Thus $R$ is symmetric.
Also since $R_s$ is a unital left $R_s R_{s^*}$-module
it follows that the ring $R_s R_{s^*}$ is left unital.

(ii)$\Rightarrow$(iii):
Take $s \in S$ and let $\epsilon_s$ denote a left unit 
for the ring $R_s R_{s^*}$. Take $r \in R_s$.
Since $R$ is symmetric there exists $n \in \mathbb{N}$
and $a_i,c_i \in R_s$ and $b_i \in R_{s^*}$,
for $i = 1 , \ldots , n$, such that $r = \sum_{i=1}^n a_i b_i c_i$.
Since $a_i b_i \in R_s R_{s^*}$ it follows that 
$\epsilon_s a_i b_i = a_i b_i$ and thus 
$\epsilon_s r = \sum_{i=1}^n \epsilon_s a_i b_i c_i = 
\sum_{i=1}^n a_i b_i c_i = r$.

(iii)$\Rightarrow$(i): Immediate.

$\bullet$ $\mathcal{P}$ = s-unital:

(i)$\Rightarrow$(ii):
Take $s \in S$. The symmetric part follows as in the unital case.
To deduce that the ring $R_s R_{s^*}$ is left s-unital
we use Proposition \ref{tominaga}. 

(ii)$\Rightarrow$(iii):
Take $s \in S$ and $r \in R_s$.
Since $R$ is symmetric there exists $n \in \mathbb{N}$
and $a_i,c_i \in R_s$ and $b_i \in R_{s^*}$,
for $i = 1 , \ldots , n$, such that $r = \sum_{i=1}^n a_i b_i c_i$.
Since $a_i b_i \in R_s R_{s^*}$ it follows, from Proposition \ref{tominaga}, 
that there is $\epsilon_s \in R_s R_{s^*}$ such that
$\epsilon_s a_i b_i = a_i b_i$, for $i=1,\ldots,n$. 
Thus $\epsilon_s r = \sum_{i=1}^n \epsilon_s a_i b_i c_i = 
\sum_{i=1}^n a_i b_i c_i = r$.

(iii)$\Rightarrow$(i): Immediate.
\end{proof}

\begin{cor}\label{epsiloncor}
If $\mathcal{P}$ denotes either ''unital'' or ''s-unital'',
then (i), (ii) and (iii) below are equivalent.
\begin{itemize}

\item[(i)] $R$ is $\mathcal{P}$ epsilon-strong.

\item[(ii)] $R$ is symmetric and for all $s \in S$, the ring
$R_s R_{s^*}$ is $\mathcal{P}$.

\item[(iii)] 
\begin{itemize}

\item[$\bullet$] $\mathcal{P}$ = unital:
for all $s \in S$, there exists $\epsilon_s \in R_s R_{s^*}$
such that for all $r \in R_s$,
$\epsilon_s r = r \epsilon_{s^*} = r$.

\item[$\bullet$] $\mathcal{P}$ = s-unital:
for all $s \in S$ and all $r \in R_s$, there exists 
$\epsilon_s \in R_s R_{s^*}$ and $\epsilon_s' \in R_{s^*} R_s$
with $\epsilon_s r = r \epsilon_s' = r$.

\end{itemize}
\end{itemize}
\end{cor}

\begin{proof}
The case $\mathcal{P}$ = unital follows from Proposition \ref{epsilonequivalent}
if we note that the ring $R_s R_{s^*}$ is unital and hence has a unique
multiplicative identity $\epsilon_s$. Then the unique multiplicative identity
of $R_{s^*} R_s$ is $\epsilon_{s^*}$.
The case $\mathcal{P}$ = s-unital follows immediately from Proposition \ref{epsilonequivalent}.
\end{proof}

\begin{defi}
We say that $R$ is {\it coherent} if for all $s,t \in S$
with $s \leq t$, the inclusion $R_s \subseteq R_t$ holds. 
In that case $R$ is idempotent coherent, since for all $e \in E(S)$
and all $s \in S$ we have that $es \leq s$ and $se \leq s$
(see \cite[Section 2]{beuter2017}), 
and thus $R_e R_s \subseteq R_{es} \subseteq R_s$ and
$R_s R_e \subseteq R_{se} \subseteq R_s$.
\end{defi}

\begin{prop}\label{minimalprop}
If $R$ is coherent and left (right) $s$-unital epsilon-strong,
then $R$ is left (right) non-degenerate and
$R / C_R( Z(R_0) )$ has the ideal intersection property.
\end{prop}

\begin{proof}
We only show the ''left'' part of the proof. 
The ''right'' part can be shown in a similar way and is therefore left to the reader.
Take an $I$-minimal element $r$ and put $d(r)=n$.
Take $s_1,\ldots,s_n$ and non-zero $r_i \in R_{s_i}$, for $i = 1 , \ldots , n$,
with $r = \sum_{i=1}^n r_i$.
From Proposition \ref{epsilonequivalent} it follows that there exists
$\epsilon_{s_1} \in R_{s_1} R_{s_1^*}$ such that 
$\epsilon_{s_1} r_1 = r_1$.
Then $I \ni r - \epsilon_{s_1} r = \sum_{i=2}^n (r_i - \epsilon_{s_1} r_i)$.
Since $R$ is idempotent coherent and $s_1 s_1^* \in E(S)$ it follows that
$r_i - \epsilon_{s_1} r_i \in R_{s_i}$, for $i = 1,\ldots,n$.
Thus $d( r - \epsilon_{s_1} r ) < n$. From $I$-minimality
it follows that  $r = \epsilon_{s_1} r$.
Since $\epsilon_{s_1} \in R_{s_1} R_{s_1^*}$ it follows in particular
that there is $z \in R_{s_1^*}$ with $z r$ non-zero.
From Proposition \ref{propintersection} it follows that $R / C_R( Z(R_0) )$ has the ideal intersection property.
\end{proof}

Now we can state and prove the main result of this section
(which in Section \ref{introduction} was named Theorem \ref{maintheorem}).

\begin{thm}\label{newmaintheorem}
If $R$ is a system simple cohe\-rent left (or right) s-unital epsilon-strong system
and $C_R( Z(R_0) ) \subseteq R_0$, then $R$ is simple.
\end{thm}

\begin{proof}
This follows from Theorem \ref{ZR0MAX} and Proposition \ref{minimalprop}.
\end{proof}

\begin{cor}\label{corrcomm2}
If $R$ is a coherent s-unital epsilon-strong system and $R_0$ is maximally commutative in $R$,
then $R$ is simple if and only if $R$ is system simple.
\end{cor}

\begin{proof}
This follows from Remark \ref{simpleremark} and Theorem \ref{newmaintheorem}.
\end{proof}

\section{Skew inverse semigroup rings}\label{partialskewsection}

In this section, we recall the definition of
skew inverse semigroup rings and we state some well known
facts concerning them that we need in the sequel.
We determine when such rings are left (or right) epsilon-strong systems
(see Proposition \ref{associative}).
At the end of this section, we prove Theorem \ref{secondmaintheorem}
(see Theorem \ref{newsecondmaintheorem}). 
Throughout this section $A$ denotes an associative, 
but not necessarily unital, ring,
$S$ is an inverse semigroup and
$\pi$ is a {\it partial action} of $S$ on $A$.
Recall that the latter means that there is a set
$\{ D_s \}_{s \in S}$ of ideals of $A$ and a set
$\{ \pi_s : D_{s^*} \rightarrow D_s \}_{s \in S}$
of ring isomorphisms satisfying the following
assertions for all $s,t \in S$:
\begin{itemize}

\item[(i)] $A = \sum_{s \in S} D_s$;

\item[(ii)] $\pi_s( D_{s^*} \cap D_t ) = D_s \cap D_{st}$;

\item[(iii)] for all $x \in D_{t^*} \cap D_{t^* s^*}$
the equality $\pi_s ( \pi_t (x) ) = \pi_{st}(x)$ holds.

\end{itemize}
We say that $\pi$ is {\it unital (locally unital, left s-unital, right s-unital)}
if for every $s \in S$ the ring $D_s$ is unital (locally unital,
left s-unital, right s-unital).
Recall that an ideal $J$ of $A$ is called {\it $S$-invariant} if
for all $s \in S$ the inclusion $\pi_s ( J \cap D_{s^*} ) \subseteq J$ holds.
The ring $A$ is called {\it $S$-simple} if $\{ 0 \}$ and $A$ are the 
only $S$-invariant ideals of $A$.
Note that if $\pi$ is a left (right) s-unital partial action of $S$ on $A$,
then for all $s \in S$ and all ideals $J$ of $A$ the equality
$J \cap D_s = D_s J$ ($J \cap D_s = J D_s$) holds.
Now we will recall the definition of the skew inverse semigroup ring 
$A \rtimes_{\pi} S$ defined by $\pi$.
Let $L_{\pi}$ be the set of formal finite sums of elements of the form $a_s \delta_s$,
for $s \in S$ and $a_s \in D_s$. We equip $L_{\pi}$ with 
component-wise addition and with a multiplication defined by the
additive extension of the relations
$( a_s \delta_s ) ( b_t \delta_t ) = 
\pi_s ( \pi_{s^*} (a_s) b_t ) \delta_{st},$
for $s,t \in S$, $a_s \in D_s$ and $b_t \in D_t$.
Let $I$ be the ideal of $L_{\pi}$ generated by all elements of
the form $a \delta_r - a \delta_s$, for $r,s \in S$
with $r \leq s$ and $a \in D_r$.
The {\it skew inverse semigroup ring} $A \rtimes_{\pi} S$ is defined 
to be the quotient $L_{\pi} / I$.
It is clear that $L_{\pi}$ is a graded ring and that $A \rtimes_{\pi} S$ is a system.
Note that the product on $L_{\pi}$ is not in general associative.
However, as we shall soon see, in the s-unital case, this is indeed so.

\begin{prop}
The ring $L_{\pi}$, and hence also the ring $A \rtimes_{\pi} S$, is a coherent system.
\end{prop}

\begin{proof}
Take $s,t \in S$ with $s \leq t$.
From \cite[Proposition 2.2]{beuter2017} it follows
that $D_s \subseteq D_t$. Thus $D_s \delta_s \subseteq D_t \delta_t$
and, hence $R_s = \overline{D_s \delta_s} \subseteq \overline{D_t \delta_t} = R_t$.  
\end{proof}

\begin{prop}\label{DsDs}
Put $R = L_{\pi}$ and take $s \in S$. The equality 
$R_s R_{s^*} = D_s D_s \delta_{s s^*}$ holds.
In particular, the ring $R_s R_{s^*}$ is left (right) s-unital 
if and only if the ring $D_s D_s$ is left (right) s-unital.
\end{prop}

\begin{proof}
We have that $R_s R_{s^*} = D_s \delta_s D_{s^*} \delta_{s^*} = 
\pi_s( \pi_{s^*}(D_s) D_{s^*} ) \delta_{s s^*} =
\pi_s( D_{s^*} D_{s^*} ) \delta_{s s^*}$ 
$=D_s D_s \delta_{s s^*}.$
From \cite[Proposition 2.2]{beuter2017} it follows that
$D_s D_s \subseteq D_s \subseteq D_{s s^*}$ and
$\pi_{s s^*} = {\rm id}_{D_{s s^*}}$, thus the last claim follows.
\end{proof}

\begin{prop}\label{DsDsDs}
Put $R = L_{\pi}$. For all $s \in S$
the equalities $( R_s R_{s^*} ) R_s = R_s ( R_{s^*} R_s )$ 
$= ( D_s D_s D_s ) \delta_s$ hold.
In particular, $R$ is symmetric if and only if
for all $s \in S$ the ring $D_s$ is idempotent. 
\end{prop}

\begin{proof}
Take $s \in S$. Then 
$$( R_s R_{s^*} ) R_s = 
( D_s \delta_s D_{s^*} \delta_{s^*} ) D_s \delta_s =
( \pi_s( \pi_{s^*}(D_s) D_{s^*} ) ) \delta_{ss^*} D_s \delta_s =$$
$$ ( \pi_s ( D_{s^*} D_{s^*} ) ) \delta_{ss^*} D_s \delta_s =
D_s D_s \delta_{ss^*} D_s \delta_s =
\pi_{ss^*}( \pi_{ss^*}( D_s D_s ) D_s ) \delta_s =$$
$$\pi_{ss^*} ( D_s D_s D_s ) \delta_s = ( D_s D_s D_s ) \delta_s.$$
Note that $\pi_{ss^*} = {\rm id}_{D_{ss^*}}$. And
$$R _s (R_{s^*} R_s) = 
D_s \delta_s ( D_{s^*} \delta_{s^*} D_s \delta_s ) =
D_s \delta_s ( \pi_{s^*}( \pi_s(D_{s^*}) D_s ) \delta_{s^* s} ) =$$
$$D_s \delta_s \pi_{s^*} ( D_s D_s ) \delta_{s^* s} = 
D_s \delta_s D_{s^*} D_{s^*} \delta_{s^* s} =
\pi_s ( \pi_{s^*}(D_s) D_{s^*} D_{s^*} ) \delta_s =$$  
$$ \pi_s ( D_{s^*} D_{s^*} D_{s^*} ) \delta_s = 
D_s D_s D_s \delta_s.$$
Now we show the last part.
Suppose first that $R$ is symmetric. Then, from the above, it follows that
$D_s = D_s D_s D_s \subseteq D_s D_s \subseteq D_s$. Hence $D_s$ is idempotent.
Now suppose that $D_s$ is idempotent. Then
$D_s D_s D_s = D_s D_s = D_s$. Thus, from the above, it
follows that $R$ is symmetric.
\end{proof}

\begin{prop}\label{associative}
The ring $L_{\pi}$ is a left (right) s-unital epsilon-strong system
if and only if $\pi$ is left (right) s-unital. 
In that case, $L_{\pi}$, and hence also $A \rtimes_{\pi} S$, is associative.
\end{prop}

\begin{proof}
The ''if'' statement follows from Proposition \ref{epsilonequivalent},
Proposition \ref{DsDs} and Proposition \ref{DsDsDs}. 
Now we show the ''only if'' statement. Take $s \in S$. 
From Proposition \ref{DsDsDs} it follows that $D_s D_s= D_s$.
From Proposition \ref{DsDs} we get that the ring $D_s D_s$
is left (right) s-unital.
This, in combination with the equality $D_s D_s = D_s$,
implies that $D_s$ is left (right) s-unital as a $D_s D_s$-module.
Therefore, in particular, $D_s$ is left (right) s-unital as a ring.
For the last statement, suppose that $D_s$ is left (right) s-unital.
We wish to show that $R$ is associative.
From \cite[Theorem 3.4]{exel2010} this follows if we can
show the equality $a \pi_s ( \pi_{s^*} (b) c ) = \pi_s ( \pi_{s^*}(ab) c )$
for all $r,s,t \in S$, all $a \in D_{r^*}$,
all $b \in D_s$ and all $c \in D_t$.
First we show the ''left'' part.
Since $D_s$ is left s-unital, there exists $d \in D_s$ such that 
$d \pi_s ( \pi_{s^*} (b) c ) = \pi_s ( \pi_{s^*} (b) c ) $ and $d b = b$.
Then 
$a \pi_s ( \pi_{s^*} (b) c ) = 
a d \pi_s ( \pi_{s^*} (b) c ) = 
\pi_s ( \pi_{s^*} (ad) ) \pi_s ( \pi_{s^*} (b) c ) =$$
$$\pi_s ( \pi_{s^*} (ad) \pi_{s^*} (b) c ) =
\pi_s ( \pi_{s^*} (adb) c ) =
\pi_s ( \pi_{s^*}(ab) c ).$
Now we show the ''right'' part.
Since $D_s$ is right s-unital, there exists $e \in D_{s^*}$ such that 
$\pi_{s^*}(b)e = b$ and $\pi_{s^*}(ab)e = \pi_{s^*}(ab)$.
Then 
$a \pi_s ( \pi_{s^*} (b) c ) = 
a \pi_s ( \pi_{s^*} (b) e c ) =
a \pi_{ss^*}(b) \pi_s( ec ) =
ab \pi_s (ec) = \pi_{s s^*} (ab) \pi_s(ec) =
\pi_s( \pi_{s^*}(ab) ec ) = 
\pi_s( \pi_{s^*}(ab) c ).$
\end{proof}

\begin{rem}\label{subring}
In \cite[Proposition 3.1]{beuter2018} it is shown that 
if $\pi$ and $A$ are locally unital, then the map 
$i : A \rightarrow (A \rtimes_{\pi} S )_0$ defined by sending
$a = \sum_{i=1}^n a_{e_i}$, where $a_{e_i} \in D_{e_i}$,
to $\sum_{i=1}^n \overline{a_{e_i} \delta_{e_i}}$,
is a well defined isomorphism of rings
with inverse given by the restriction to $(A \rtimes_{\pi} S )_0$
of the map $t : A \rtimes_{\pi} S \rightarrow A$
defined by $t( \sum_{i=1}^n \overline{a_i \delta_{s_i} } ) = \sum_{i=1}^n a_i$,
for $s_i \in S$ and $a_i \in D_{s_i}$.
It is clear from the proof given in loc. cit.
that the same conclusions hold when $\pi$ and $A$ are s-unital.
\end{rem}

\begin{prop}\label{propsystemsimple}
If $\pi$ and $A$ are s-unital, then
$A \rtimes_{\pi} S$ is system simple if and only if $A$ is $S$-simple.
\end{prop}

\begin{proof}
Put $R = A \rtimes_{\pi} S$.

First we show the ''only if'' statement.
Suppose that $A \rtimes_{\pi} S$ is system simple.
Let $J$ be a non-zero $S$-invariant ideal of $A$.
For all $s \in S$ put $I_s = \overline{(J \cap D_s) \delta_s}$
and let $I = \sum_{s \in S} I_s$.
Take $s \in S$. Since $I_s \subseteq R_s$ it follows that
$I_s \subseteq I \cap R_s$. Thus $I \subseteq \sum_{s \in S} I \cap R_s$.
The inclusion $I \supseteq \sum_{s \in S} I \cap R_s$ is trivial.
Thus $I = \sum_{s \in S} I \cap R_s$.
Now we show that $I$ is an ideal of $R$.
To this end, take $s,t \in S$. 
Then 
$$\overline{ D_t \delta_t } \cdot \overline{ (J \cap D_s) \delta_s } =
\overline{ \pi_t( \pi_{t^*}(D_t)  (J \cap D_s)  ) \delta_{ts} } = 
\overline{ \pi_t( D_{t^*}  (J \cap D_s)  ) \delta_{ts} }.$$
Since $J$ is $S$-invariant, we get that
$\pi_t( D_{t^*}  (J \cap D_s)  ) \subseteq 
\pi_t( D_{t^*} \cap J ) \subseteq J$
and from the definition of partial action, we get that
$\pi_t( D_{t^*}  (J \cap D_s)  ) \subseteq 
\pi_t( D_{t^*} \cap D_s ) = D_t \cap D_{ts} \subseteq D_{ts}.$
Thus $I$ is a left ideal of $R$. Also 
$$\overline{ (J \cap D_s) \delta_s } \cdot \overline{ D_t \delta_t } =
\overline{ \pi_s( \pi_{s^*}(J \cap D_s)  D_t  ) \delta_{st} } = 
\overline{ \pi_s( J \cap D_{s^*} \cap D_t ) \delta_{ts} }.$$
Since $J$ is $S$-invariant, we get that
$\pi_s( J \cap D_{s^*} \cap D_t ) \delta_{ts} \subseteq 
\pi_s( J \cap D_{s^*} ) \subseteq J$
and from the definition of partial action, we get that
$\pi_s( J \cap D_{s^*} \cap D_t ) \delta_{ts} \subseteq 
\pi_s( D_{s^*} \cap D_t ) =  D_s \cap D_{st} \subseteq D_{st}.$
Thus $I$ is a right ideal of $R$.
Since $R$ is system simple, this implies that $I = R$.
Then $t(I) = t(R) = \sum_{s \in S} D_s = A$.
On the other hand, from the construction of $I$, it follows that
$t(I) = \sum_{s \in S} (J \cap D_s) \subseteq J$.
Thus $J = A$.

Now we show the ''if'' statement.
Suppose that $A$ is $S$-simple.
Let $I$ be a non-zero system ideal of $R$.
We wish to show that $I = R$.
Then $t(I)$ is a non-zero additive subgroup of $A$.
First we show that $t(I)$ is an ideal of $A$.
Take $s,t \in S$, $a \in D_s$ and $b \in D_t$
such that $\overline{ b \delta_t } \in I$. 
Since $D_t$ is s-unital, there is $c \in D_t$ such that $cb = bc = b$. Then 
$ab = 
t( \overline{ ab \delta_t } ) = 
t( \overline{ acb \delta_t } ) = 
t( \overline{ \pi_{tt^*}(\pi_{tt^*}(ac)b) \delta_{tt^*t} } ) =
t( \overline{ac \delta_{tt^*}} \cdot \overline{b \delta_t}) \in t(I)$
and
$ba = 
t( \overline{ bca \delta_t} ) =
t( \overline{ \pi_t( \pi_{t^*}(b) \pi_{t^*}(ca) ) \delta_t } ) =
t( \overline{b \delta_t} \cdot \overline{ \pi_{t^*}(ca) \delta_{t^*t} } ) \in t(I).$
Now we show that $t(I)$ is $S$-invariant.
Take $s \in S$ and $a \in t(I) \cap D_{s^*}$.
There exists $n \in \mathbb{N}$, $t_n,\ldots,t_n \in S$,
$b_i \in D_{t_i}$, for $i = 1,\ldots,n$,
such that $\sum_{i=1}^n \overline{ b_i \delta_{t_i} } \in I$
and $\sum_{i=1}^n b_i = t( \overline{ \sum_{i=1}^n b_i \delta_{t_i} } ) = a$.
Since $D_{s^*}$ is s-unital there is $d \in D_s$ such that
$\pi_{s^*}(d) a = a$. Then
$\pi_s ( a ) = 
\pi_s( \pi_{s^*}(d) a ) =
\sum_{i=1}^n \pi_s ( \pi_{s^*}(d) b_i ) =
\sum_{i=1}^n t( \overline{ \pi_s ( \pi_{s^*}(d)b_i ) \delta_{s t_i} } )
= t( \overline{d \delta_s} \cdot \sum_{i=1}^n \overline{ b_i \delta_{t_i} } ) \in t(I).$
Thus $t(I)$ is a non-zero $S$-invariant ideal of $A$.
From $S$-simplicity of $A$, we hence get that $t(I) = A$.
Take $s \in S$ and $a \in D_s$. We wish to show that $\overline{a \delta_s} \in I$.
Since $D_s$ is s-unital there is $b \in D_s$ such that $ba = a$.
Since $t(I) = A$, there is $n \in \mathbb{N}$, 
$s_n,\ldots,s_n \in S$,
$b_i \in D_{s_i}$, for $i = 1,\ldots,n$,
such that $\sum_{i=1}^n \overline{ b_i \delta_{s_i} } \in I$
and $\sum_{i=1}^n b_i = t( \sum_{i=1}^n \overline{ b_i \delta_{s_i} } ) = b$.
Since $I$ is a system ideal we may assume that $\overline{ b_i \delta_{s_i} } \in I$
for $i = 1, \ldots , n$. Take $i \in \{ 1, \ldots , n\}$.
Since $D_{s_i^*}$ is s-unital there is $c \in D_{s_i^*}$
such that $\pi_{s_i^*}(b_i)c = \pi_{s_i^*}(b_i)$. Then
$I \ni \overline{b_i \delta_{s_i} } \cdot \overline{c \delta_{s_i^*}} =
\overline{ \pi_{s_i} ( \pi_{s_i^*}(b_i) c ) \delta_{s_i s_i^*} } =
\overline{ \pi_{s_i} ( \pi_{s_i^*}(b_i) ) \delta_{s_i s_i^*} } =
\overline{ b_i \delta_{s_i s_i^*} }.$
Therefore, we may assume that all $s_i$ are idempotent.
From \cite[Proposition 2.2]{beuter2018} it follows that
$\pi_{s_i} = {\rm id}_{D_{s_i}}$. Thus, from the definition of $\pi$
we get that $D_{s_i} D_s \subseteq D_{s_i} \cap D_s =
\pi_{s_i} ( D_{s_i^*} \cap D_s ) = D_{s_i} \cap D_{s_i s}
\subseteq D_{s_i s}$.
Hence, finally, we get that
$\overline{ a \delta_s } =
\overline{ba \delta_s} =
\overline{ \sum_{i=1}^n b_i a \delta_s } =
[s_i s \leq s] =
\overline{ \sum_{i=1} ^n b_i a \delta_{s_i s}} =
\overline{ ( \sum_{i=1}^n b_i \delta_{s_i} ) a \delta_s  } =
\overline{ \sum_{i=1}^n b_i \delta_{s_i} } \cdot \overline{ a \delta_s  } \in I$.
\end{proof}

\begin{prop}\label{propequivalence}
If $\pi$ is $s$-unital, and $A$ is commutative and $s$-unital, 
then $A$ is a maximal commutative subring of $A \rtimes_{\pi} S$ if and only if
$( A \rtimes_{\pi} S ) / A$ has the ideal intersection property.
\end{prop}

\begin{proof}
The ''only if'' statement follows from Proposition \ref{minimalprop} and Proposition \ref{associative}.
The ''if'' statement follows from the first part of the proof of \cite[Theorem 3.4]{beuter2017}
which holds in the s-unital case also.
\end{proof}

We are now ready to state and prove the main result of this section (which in 
Section \ref{introduction} was named Theorem \ref{secondmaintheorem}).

\begin{thm}\label{newsecondmaintheorem}
Suppose that $\pi$ is an s-unital partial action of 
an inverse semigroup $S$ on an associative (but not necessarily commutative) 
$s$-unital ring $A$.
If $A$ is $S$-simple and $C_{A \rtimes_{\pi} S}( Z(A) ) \subseteq A$, 
then $A \rtimes_{\pi} S$ is simple.
If $A$ is commutative, then $A \rtimes_{\pi} S$ is simple if and only if
$A$ is $S$-simple and $A$ is a maximal 
commutative subring of $A \rtimes_{\pi} S$.
\end{thm}

\begin{proof}
The first part follows from Theorem \ref{newmaintheorem} and Proposition \ref{propsystemsimple}.
The second part follows from the first part, Proposition \ref{propsystemsimple} and Proposition \ref{propequivalence}.
\end{proof}

\begin{rem}
Since the class of s-unital rings properly contains the class of locally unital rings
(even in the commutative case, see Example \ref{notlocallyunital}) it follows that
Theorem \ref{newsecondmaintheorem} is a proper generalization of Theorem \ref{thmbeuter}. 
\end{rem}

For use in subsequent sections, we now introduce a generalization of the concept 
of a faithful group action (see e.g. \cite[Chapter 1.4]{karpilovsky1987}).

\begin{defi}
We say that $\pi$ is {\it faithful} if for all $s \in S \setminus E(S)$,
$\pi_s \neq {\rm id}_{D_{s^*}}$ holds.
\end{defi}

\begin{prop}\label{faithfulprop}
Suppose that $\pi$ and $A$ are s-unital and that for every $s \in S \setminus E(S)$
the ring $D_s$ is non-zero.
If $C_{A \rtimes_{\pi} S}( Z(A) ) \subseteq A$, then $\pi$ is faithful. 
\end{prop}

\begin{proof}
Suppose that $\pi$ is not faithful.
Take $s \in S \setminus E(S)$ such that $\pi_s = {\rm id}_{D_{s^*}}$.
Take a non-zero $a \in D_s$ and put 
$x = \overline{ a \delta_s } - \overline{ a \delta_{ss^*} } \in ( A \rtimes_{\pi} S ) \setminus A$.
We wish to show that $x \in C_{A \rtimes_{\pi} S}( Z(A) )$.
To this end, take $b = \sum_{i=1}^n b_i \in Z(A)$,
for some $b_i \in D_{e_i}$, $e_i \in E(S)$, for $i=1,\ldots,n$. 
From Remark \ref{subring} we know that $\sum_{i=1}^n  \overline{ b_i \delta_{e_i} }$ 
commutes with $\overline{ a \delta_{ss^*} }$.
Therefore, we only need to show that $\sum_{i=1}^n  \overline{ b_i \delta_{e_i} }$ 
commutes with $\overline{ a \delta_s }$. Now
$$\sum_{i=1}^n  \overline{ b_i \delta_{e_i} } \cdot \overline{ a \delta_s } =
\sum_{i=1}^n  \overline{ b_i a \delta_{e_i s} } = [e_i s \leq s] =
\sum_{i=1}^n  \overline{ b_i a \delta_s } = 
\overline{b a \delta_s} = [b \in Z(A)] =
\overline{ab \delta_s}$$ and
$$ \sum_{i=1}^n \overline{ a \delta_s } \cdot \overline{ b_i \delta_{e_i} } =
\sum_{i=1}^n \overline{ \pi_s( \pi_{s^*}(a) b_i ) \delta_{s e_i} } =
[s e_i \leq s, \ \pi_s = {\rm id}_{D_{s^*}} ] = 
\sum_{i=1}^n \overline{ a b_i \delta_s} = \overline{ab \delta_s}.$$
\end{proof}

\begin{prop}\label{faithfulsimpleprop}
Suppose that $\pi$ and $A$ are s-unital and that for every $s \in S \setminus E(S)$
the ring $D_s$ is non-zero.
If $A \rtimes_{\pi} S$ is simple, then $\pi$ is faithful. 
\end{prop}

\begin{proof}
Suppose that $\pi$ is not faithful.
Then there is $s \in S \setminus E(S)$ such that $\pi_s = {\rm id}_{D_{s^*}}$.
Take a non-zero $a \in D_s$ and consider the non-zero element
$x = \overline{ a \delta_s } - \overline{ a \delta_{ss^*} }$ in $A \rtimes_{\pi} S$.
Let $I$ denote the non-zero ideal in $A \rtimes_{\pi} S$ generated by $x$. 
We claim that $t(I) = 0$. If we assume that the claim holds,
then it follows that $I$ is a proper ideal of $A \rtimes_{\pi} S$,
since e.g. $t( \overline{ a \delta_s } ) = a \neq 0$,
and thus $A \rtimes_{\pi} S$ is not simple.
Now we show the claim.
Take $r,t \in S$, $b \in D_r$ and $c \in D_t$. Then
$$\overline{b \delta_r} \cdot x = 
\overline{ \pi_r ( \pi_{r^*}(b) a ) \delta_{rs} } - \overline{ \pi_r ( \pi_{r^*}(b) a ) \delta_{rs s^*} }$$
and
$$x \cdot \overline{c \delta_t} = 
\overline{ a c \delta_{st} } - \overline{ ac \delta_{s s^* t} } $$
and
$$\overline{b \delta_r} \cdot x \cdot \overline{c \delta_t} =
\overline{ \pi_r ( \pi_{r^*}(b) ac ) \delta_{rst} } - \overline{ \pi_r ( \pi_{r^*}(b) ac ) \delta_{rs s^* t} } $$
from which the claim follows.
\end{proof}

\section{Steinberg algebras}\label{sectionsteinbergalgebras}

In this section, we recall from \cite{beuter2017}
the description of Steinberg algebras as skew inverse 
semigroup rings. We use this description and the previous results
to prove Theorem \ref{thirdmaintheorem} (see Theorem \ref{newthirdmaintheorem}).
At the end of this section, we specialize this result to the
case when the topology on the groupoid is discrete (se Theorem \ref{fourthmaintheorem}).
Note that we closely follow the presentation from 
\cite{beuter2017}, in particular in the proofs
of Propositions \ref{gen2} and Proposition \ref{gen1}.

Let $G$ be a {\it groupoid}. By this we mean that $G$ is a small category in which every 
morphism is an isomorphism. The objects of $G$ will
be denoted by $G_0$ and the the set of morphisms of $G$
will be denoted by $G_1$.
The {\it domain} and {\it codomain} of $g \in G_1$ will 
be denoted by $d(g)$ and $c(g)$ respectively.
Objects will be identified with the corresponding units
so that in particular, for all $g \in G_1$,
the identities $d(g) = g^{-1} g$ and $c(g) = g g^{-1}$ hold.
Let $G_2$ denote the set of {\it composable pairs} of $G_1$,
that is, all $(g,h) \in G_1 \times G_1$ such that $d(g) = c(h)$.
We say that $G$ is a {\it topological groupoid} if $G_1$ 
is a topological space making inversion and composition
continuous as maps $G_1 \rightarrow G_1$ and
$G_2 \rightarrow G_1$, respectively, where the set $G_2$
is equipped with the relative topology induced from 
the product topology on $G_1 \times G_1$.
A {\it bisection} of $G$ is a subset $U$ of $G_1$
such that both $c|_U$ and $d|_U$ are homeomorphisms. 
We call $G$ {\it \'{e}tale} if 
$G_0$ is locally compact and Hausdorff, and 
$d$ is a local homeomorphism (in that case, $c$ is also a local homeomorphism).
An \'{e}tale groupoid $G$ is said to be {\it ample} if $G_1$ has a basis 
of compact open bisections. One can show that a Hausdorff \'{e}tale groupoid
is ample if and only if $G_0$ is totally disconnected.
A subset $U$ of $G_0$ is called {\it invariant} if $d(c^{-1}(U)) = U$.
The groupoid $G$ is called {\it minimal} if $G_0$ and $\emptyset$ are the only
open invariant subsets of $G_0$. We let ${\rm Iso}(G)$ denote the 
isotropy subgroupoid of $G$, that is the set of all $g \in G_1$ with $d(g) = c(g)$.
If $G$ is Hausdorff and ample, then $G$ is said to be {\it effective} if
the interior of ${\rm Iso}(G)$ is $G_0$, or equivalently,
for all compact open bisections $U$ of $G_1 \setminus G_0$, there 
exists $a \in U$ such that $d(a) \neq c(a)$.
We let $G^a$ denote the set of compact open bisections of $G_1$.
The set $G^a$ is an inverse semigroup under the operations defined by
$UV = \{ gh \mid g \in U, \ h \in V, (g,h) \in G_2 \}$ and
$U^* = U^{-1} = \{ a^{-1} \mid a \in U \}$, for $U,V \in G^a$.
The inverse semigroup partial order in $G^a$ is the inclusion of sets.
The set of idempotents $E(G^a)$ is given by the set of $U \in G^a$
such that $U \subseteq G_0$.
From now on we assume the following:
\begin{itemize}
\item $K$ is a non-zero associative (but not necessarily commutative or unital) ring
and $G$ is a Hausdorff and ample groupoid. 
\end{itemize}
The {\it Steinberg algebra}
$A_K(G)$ is defined to be the set of compactly supported locally 
constant functions from $G_1$ to $K$ with pointwise addition,
and convolution product, that is, if $f,g \in A_K(G)$
and $a \in G_1$, then $(f * g)(a) = \sum_{bc=a} f(b)g(c)$.
In \cite{steinberg2010} it is shown that $A_K(G)$ is 
associative in the case when $K$ is a commutative unital ring.
It is clear that the associativity $A_K(G)$ also holds 
for general associative rings $K$.
If $k \in K$ and $U \in G^a$, then we let $k_U : G_1 \rightarrow K$
denote the function defined by $k_U(a) = k$, if $a \in U$,
and $k_U(a)=0$, otherwise.
The algebra $A_K(G)$ can be realised as the additive span of
functions of the form $k_U$, for $U \in G^a$.
Convolution of such functions is nicely behaved
in the sense that $k_U * l_V = (kl)_{UV}$,
for $k,l \in K$ and $U,V \in G^a$.
The product of two subsets $U$ and $V$ of $G_1$ is defined 
as $UV = \{ gh \mid g \in U, \ h \in V, (g,h) \in G_2 \}$.

Now we will describe the translation of Steinberg algebras
to skew inverse semigroup rings from \cite{beuter2017}.
From now on, let $G$ be a Hausdorff and ample groupoid.
Given $U \in G^a$, define a map
$\theta_U : d(U) \rightarrow c(U)$ by
$\theta_U(u) = c_U ( d_U^{-1}(u) )$, for $u \in U$.
Here $c_U$ and $d_U$ denote the corresponding restrictions of the
the maps $c$ and $d$.
The correspondence $U \mapsto \theta_U$ gives a partial action
of $G^a$ on $G_0$.
Let $L_c( G^0 )$ denote the ring of all locally constant,
compactly supported, $K$-valued functions on $G_0$.
Given $U \in G^a$, let $D_U$ denote the ideal 
$\{ f \in L_c(G_0) \mid \mbox{if $x \in G_0 \setminus c(U)$, then $f(x)=0$} \} = L_c(c(U))$
of $L_c(G_0)$ and define a ring isomorphism
$\pi_U : D_{U^*} \rightarrow D_U$ in the following way.
If $f \in D_{U^*}$ and $x \in c(U)$, then let 
$\pi_U(f)(x) = f ( \theta_{U^*}(x) )$, if $x \in c(U)$,
and $\pi_U(f)(x) = 0$, otherwise.
Define the map $\alpha : L_c(G_0) \rtimes_{\pi} G^a \rightarrow A_K(G)$ by
$\alpha( \overline{f \delta_B} )(x) = f(c(x))$, if $x \in B$,
and $\alpha( \overline{f \delta_B} )(x) = 0$, otherwise.
Define the map $\beta : A_K(G) \rightarrow  L_c(G_0) \rtimes_{\pi} G^a$
in the following way. Let $f = \sum_{j=1}^n (k_j)_{B_j} \in A_K(G)$,
where the $B_j$ are pairwise disjoint compact bisections of $G$.
Then let $\beta(f) = \sum_{j=1}^n \overline{ (k_j)_{c(B_j)} \delta_{B_j} }$.
From \cite[Theorem 5.2]{beuter2018} it follows that
$\alpha \circ \beta = {\rm id}_{A_K(G)}$ and
$\beta \circ \alpha = {\rm id}_{ L_c(G_0) \rtimes_{\pi} G^a }$.

\begin{defi}
Let $T$ denote $(L_c(G_0) \rtimes G^a)_0$, that is the set of all 
finite sums of elements in $L_c(G_0) \rtimes G^a$ of the form 
$\overline{g \delta_U}$, for $U \in E(G^a)$ and $g \in L_c(U)$.
\end{defi}

Now we will describe some topological properties of $G$ in terms of
algebraical properties of $L_c(G_0) \rtimes_{\pi} G^a$.
We first consider effectiveness.

\begin{prop}\label{effectiveprop}
The groupoid $G$ is effective if and only if $\pi$ is faithful.
\end{prop}

\begin{proof}
Suppose that $G$ is not effective. Then there exists $U \in G^a \setminus E(G^a)$
such that for all $g \in U$, the relation $d(g) = c(g)$ holds.
Then $\theta_U = {\rm id}_{d(U)} = {\rm id}_{c(U)}$ and hence
$\pi_U = {\rm id}_{D_{U^*}}$. Thus $\pi$ is not faithful.

Now suppose that $\pi$ is not faithful. Then there exists $V \in G^a \setminus E(G^a)$
such that $\pi_V = {\rm id}_{D_{V^*}}$.
Take $g \in V^*$.
Take a non-zero $k \in K$ and define $f \in D_{V^*}$ by saying that
$f( d(g) ) = k$ and $f( a ) = 0$, for $a \in G_0 \setminus \{ d(g) \}$.
Then $k = f( d(g) ) = 
\pi_V(f)(d(g)) =
f ( \theta_{V^*} ( d(g) ) ) =
f ( c_{V^*}( d_{V^*}^{-1}( d(g) ) ) ) = 
f(c(g))$ which implies that $c(g)=d(g)$.
Thus $G$ is not effective. 
\end{proof}

\begin{prop}\label{propeffective}
If $K$ is s-unital and $C_{ L_c(G_0) \rtimes G^a }( Z(T) ) \subseteq T$,
then $G$ is effective.
\end{prop}

\begin{proof}
This follows from Proposition \ref{faithfulprop} and Proposition \ref{effectiveprop}.
\end{proof}

\begin{defi}
We say that $Z(K)$ contains a set of
s-units for $K$ if for all $k \in K$ there exists $k' \in Z(K)$
such that $k k' = k$. 
\end{defi}

\begin{prop}\label{gen2}
If $Z(K)$ contains a set of s-units for $K$, then the following 
are equivalent:
\begin{itemize}

\item[(i)] $G$ is effective; 

\item[(ii)] $\pi$ is faithful;

\item[(iii)] $C_{ L_c(G_0) \rtimes G^a }( Z(T) ) \subseteq T$.
\end{itemize}
\end{prop}

\begin{proof}
From Proposition \ref{effectiveprop} it follows that (i)$\Leftrightarrow$(ii).
The implication (iii)$\Rightarrow$(i) follows from Proposition \ref{propeffective}.
Now we show the implication (i)$\Rightarrow$(iii).
To this end, take a non-zero 
$f = \sum_{i=1}^n \overline{ (k_i)_{c(U_i)} \delta_{U_i} } \in L_c(G_0) \rtimes G^a$
where the $k_i \in K \setminus \{ 0 \}$ and the $B_i$ are pairwise
disjoint compact bisections of $G$.
Suppose that $f \in C_{ L_c(G_0) \rtimes G^a }( Z(T) )$.
We wish to show that $f \in T$.
Since $G$ is effective, it suffices to show that each $B_i$
is a subset of ${\rm Iso}(G)$.
Seeking a contradiction, suppose that there is 
$l \in \{ 1, \ldots , n \}$ and $b \in B_l$ such that $c(b) \neq d(b)$.
From the Hausdorff property of $G$ it follows that there is a
compact bisection $U \subseteq G_0$ with $c(b) \in U$ but $d(b) \notin U$.
Since $Z(K)$ contains a set of s-units for $K$, there exists $\epsilon \in Z(K)$
such that for all $i \in \{ 1,\ldots,n \}$, the equality $k_i \epsilon = k_i$ holds.
Since, clearly, $\overline{ \epsilon \delta_U } \in Z(T)$, it follows that
$\overline{ \epsilon \delta_U } f = f \overline{ \epsilon \delta_U }$.
By mimicking the calculations in the proof of 
\cite[Proposition 4.8]{beuter2017}, we get that
$(*) \ \sum_{i=1}^n (k_i)_{C_i} = \sum_{i=1}^n (k_i)_{D_i}$,
where $C_i = c_{B_i}^{-1} ( U \cap c(B_i) )$ and
$D_i = c_{B_i}^{-1}( c(B_i) \cap \theta_{B_i}( d(B_i) \cap U ) )$,
for $i = 1,\ldots,n$. By evaluating (*) on $b$,
we get the equality $(k_l)_{C_l}(b) = (k_l)_{D_l}(b)$.
This equality yields the contradiction $k_l = 0$.
\end{proof}

Next, we consider minimality.

\begin{prop}\label{minimalpropagain}
If $L_c(G_0)$ is $G^a$-simple, then $G$ is minimal.
\end{prop}

\begin{proof}
It is clear that the second part of the proof of \cite[Proposition 5.4]{beuter2017}
works in our generality also.
\end{proof}

\begin{lem}\label{thelemma}
Suppose that $K$ is an s-unital ring. 
If $I$ is an ideal of $L_c(G_0)$ such that for all $k \in K$ and
all $x \in G_0$ there exists a compact open $V \subseteq G_0$ such that 
$x \in V$ and the map $k_V$ belongs to $I$, then $I = L_c(G_0)$. 
\end{lem}

\begin{proof}
Seeking a contradiction, suppose that $I \subsetneq L_c(G_0)$.
Since every function in $L_c(G_0)$ is a sum of funtions
of the form $k_U$, for $k \in K$ and compact open subsets $U$ of $G_0$,
it follows that there must exist a non-zero $k \in K$ and a non-empty
compact open subset $U$ of $G_0$ with $k_U \notin I$.
Amongst all such maps, we may, from compactness, choose
one $k_U$ with $U$ minimal.
Since $K$ is s-unital there exists $k' \in K$ such that $kk' = k$.
Take $x \in U$.
From the assumptions it follows that there exists 
a compact open subset $V$ of $G_0$ such that $x \in V$ and $k'_V \in I$.
But then $k_{U \cap V} = k_U * k_V' \in I$. Since $x \in U \cap V$
it follows, in particular, that $U \cap V \neq \emptyset$. 
Thus, from minimality of $U$, it follows that $U \subseteq V$.
But then $k_U = k_{U \cap V} \in I$ which is a contradiction.
\end{proof}

\begin{prop}\label{gen1}
If $K$ is simple and s-unital, then $G$ is minimal if and only if 
$L_c(G_0)$ is $G^a$-simple.
\end{prop}

\begin{proof}
The ''if'' statement follows from Proposition \ref{minimalpropagain}.
Now we show the ''only if'' statement.
Let $I$ be a non-zero $G^a$-invariant ideal of $L_c(G_0)$.
Define $U = \{ u \in G_0 \mid \mbox{there exists $f \in I$ with $f(u) \neq 0$} \}.$
Then, clearly, $U$ is non-empty.
We claim that $U$ is invariant.
Assume, for a moment, that the claim holds.
From minimality of $G$ it follows that $U = G_0$.
Take $x \in G_0$ and $f \in I$ with $f(x) \neq 0$.
Then there is a compact open subset $V$ of $G_0$
with $x \in V$ such that $f|_V = f(x)_V$.
Take $k \in K$. Since $K$ is simple, there exists $n \in \mathbb{N}$
and $k_i,k_i' \in K$, for $i=1,\ldots,n$, such that
$k = \sum_{k=1}^n k_i f(x) k_i'$.
Then $k_V = \sum_{i=1}^n (r_i)_V * f * (r_i')_V \in I$.
From Lemma \ref{thelemma} it follows that $I = L_c(G_0)$.
Now we show the claim.
Let $x \in G_1$ be such that $d(x) \in U$.
Then there exists $g \in I$ with $g(d(x)) \neq 0$. 
Take a compact open bisection $V$ with $x \in V$.
Take $n \in \mathbb{N}$ and $k_1,\ldots,k_n \in K$
such that the image of $g$ equals $\{ k_1 ,\ldots , k_n \}$.
Since $K$ is s-unital there is $k \in K$ with $k_i k = k_i$,
for $i = 1,\ldots,n$. Then $h := g * k_{d(V)} \in I \cap L_c(d(V))$.
Since $I$ is $G^a$-invariant it follows that $\alpha_V(h) \in I$.
Finally, since 
$\alpha_V(h)( c(x) ) = 
h( \theta_{V^*}(c(x)) ) = 
h( d ( c_B^{-1} ( c(x) ) ) ) =
h( d ( x ) ) =
g(d(x)) \neq 0,$
we get that $c(x) \in U$ and thus that $U$ is invariant. 
\end{proof}

We are now ready to state and prove the main result of this section (which in 
Section \ref{introduction} was named Theorem \ref{thirdmaintheorem}).

\begin{thm}\label{newthirdmaintheorem}
Suppose that $K$ has the property that $Z(K)$ contains a set of s-units for $K$.
If $G$ is a Hausdorff and ample groupoid, then the Steinberg algebra $A_K(G)$
is simple if and only if $G$ is effective and minimal, and $K$ is simple.
\end{thm}

\begin{proof}
The ''if'' statement follows from Theorem \ref{newsecondmaintheorem},
Proposition \ref{gen2} and Proposition \ref{gen1}.
Now we show the ''only if'' statement.
Suppose that $K$ is not simple. Then there is
a nonzero proper ideal $J$ of $K$. Then $A_J(G)$ is a non-zero 
proper ideal of $A_K(G)$ and hence $A_K(G)$ is not simple.
If $G$ is not effective, then, from Proposition \ref{gen2} it follows that
$\pi$ is not faithful. From Proposition \ref{faithfulsimpleprop}
it thus follows that $A_K(G)$ is not simple.
Finally, suppose that $G$ is not minimal.
From Proposition \ref{gen1} we get that $L_c(G_0)$ is not $G^a$-simple.
From Propostion \ref{propsystemsimple} it now follows that $A_K(G)$ is not system simple.
Thus, from Remark \ref{simpleremark}, we get that $A_K(G)$ is not simple. 
\end{proof}

In the last part of this section, we consider the case when 
the topology on $G$ is discrete.
It is easy to see \cite[Remark 3.10]{steinberg2010} 
that the corresponding Steinberg algebra
$A_K(G)$ coincides with the classical {\it groupoid ring} $K[G]$,
of $G$ over $K$, defined in the following way.
The elements of $K[G]$ are finite sums
of formal elements of the form $k g$ where $k \in K$ and $g \in G_1$.
Addition is defined point-wise i.e. from the relations $(k g) + (k' g) = (k+k')g$,
for $k,k' \in K$ and $g \in G_1$.
Multiplication is defined by the biadditive extension
of the relations defined by $(k g)(k' g') = (kk')(gg')$,
if $d(g)=c(g')$, and $(kg)(k'g') = 0$, otherwise,
for $k,k' \in K$ and $g,g' \in G_1$.
Before we can state our result, we need an example and a definition.

\begin{exa}
Suppose that $I$ is a set.
Recall that the induced matrix groupoid $\overline{I}$, 
defined by $I$, is constructed in the following way.
Let $\overline{I}_0 = I$ and $\overline{I}_1 = I \times I$.
Given $(i,j) \in I \times I$ put $d( (i,j) ) = j$ and $c( (i,j) ) = i$.
The groupoid ring $K[\overline{I}]$ is called {\it the ring of row and column finite
matrices over $K$ defined by $I$}.
Note that if $n \in \mathbb{N}$ and 
$I = \{ 1,\ldots,n \}$, then $K[\overline{I}]$ coincides with the ring 
$M_n(K)$ of square matrices of size $n$ over $K$.
\end{exa}

\begin{defi}
The groupoid $G$ is called {\it connected (thin)} if for all $u,v \in G_0$
there exists at least (at most) one $g \in G_1$ with $d(g)=u$ and $c(g)=v$. 
\end{defi}

\begin{lem}\label{lemmaconnected}
If $G$ is a discrete groupoid, then
\begin{itemize}

\item[(a)] $G$ is minimal if and only if $G$ is connected;

\item[(b)] $G$ is effective if and only if $G$ is thin;

\item[(c)] $G$ is minimal and effective if and only if
$G$ equals the matrix groupoid defined by $G_0$.

\end{itemize}
\end{lem}

\begin{proof}
(a) Suppose that $G$ is not minimal. Then there is a nonempty
invariant $U \subsetneq G_0$ with $d(c^{-1}(U)) = U$.
Take $u \in U$ and $v \in G_0 \setminus U$.
Seeking a contradiction, suppose that there is $g \in G_1$
with $d(g)=v$ and $c(g)=u$. Then $g \in c^{-1}(U)$ so that
$v = d(g) \in d(c^{-1}(U)) = U$ which is a contradiction.
Suppose that $G$ is not connected. Take $u \in G_0$ and define
$U$ to be the non-empty set of $u' \in G_0$ such that there exists
$g \in G_1$ with $d(g)=u'$ and $c(g)=u$. Then, clearly,
$d(c^{-1}(U))=U$ and, since $G$ is not connected, $U \subsetneq G_0$.
Therefore, $G$ is not minimal.

(b) It is clear that $G$ is effective if and only if 
${\rm Iso}(G) = G_0$ and the latter is equivalent to $G$ being thin.

(c) It follows from (a) and (b) that $G$ is minimal and effective
if and only if $G$ is connected and thin. In that case, for all $u,v \in G_0$
there is precisely one $g \in G_1$ with $d(g)=u$ and $c(g)=v$.
This is equivalent to $G = \overline{G_0}$.
\end{proof}

\begin{thm}\label{fourthmaintheorem}
Suppose that $K$ is a non-zero and associative (but not necessarily commutative
or unital) ring with the property that $Z(K)$ contains a set of s-units for $K$.
If $G$ is a discrete groupoid, then the Steinberg algebra $A_K(G)$ is simple
if and only $K$ is simple and $A_K(G)$ equals the ring of row and column finite matrices over $K$
defined by the objects $G_0$ of $G$.
\end{thm}

\begin{proof}
This follows from Theorem \ref{newthirdmaintheorem} and Lemma \ref{lemmaconnected}.
\end{proof}

\begin{rem}
In the case when $K$ is locally unital, then the ''if'' statement in Theorem \ref{fourthmaintheorem} 
follows from the fact that $K$ and $A_K(G)$ are Morita equivalent (see \cite[p. 14]{anh1987}).
\end{rem}

\section{Groupoid graded rings}\label{gradedrings}

In this section, we specialize the results in the previous sections
to groupoid (and group) graded rings.
Let $G$ be a groupoid. We assume the same notation for groupoids as in Section \ref{sectionsteinbergalgebras}.
Let $R$ denote a ring. For the rest of this section, we assume that 
$R$ is {\it graded} by $G$. Recall from \cite{lundstrom2004} that this means that 
there to each $g \in G_1$ is an additive subgroup $R_g$ of $R$ such that 
$R = \oplus_{g \in G_1} R_g$ and for all $g,h \in G_1$, the
inclusion $R_g R_h \subseteq R_{gh}$ holds, if $(g,h) \in G_2$, and
$R_g R_h = \{ 0 \}$, otherwise.
Note that if $G$ only has one object, then $G$ is a group and we recover
the usual notion of a {\it group graded ring} (see e.g. \cite{cohen1983}).
Recall that an ideal $I$ of $R$ is called {\it graded} if 
$R = \oplus_{g \in G_1} R_g \cap I$; $R$ is called {\it graded simple}
if $\{ 0 \}$ and $R$ are the only graded ideals of $R$.
The grading on $R$ is called left (right) {\it non-degenerate} if for all $g \in G_1$
and all non-zero $r \in R_g$, the relation $R_{g^{-1}} r \neq \{ 0 \}$ 
($r R_{g^{-1}} \neq \{ 0 \}$) holds.
The grading on $R$ is called {\it s-unital epsilon-strongly graded}
if for all $g \in G_1$ the $R_g R_{g^{-1}}$-$R_{g^{-1}} R_g$-bimodule $R_g$ is s-unital.
Throughout this section, let $S$ denote the inverse semigroup {\it induced} by $G$.
By this we mean that $S = G_1 \cup \{ o \}$,
where $o$ is a symbol not contained in $G_1$. 
Put $o^* = o$ and if $g \in G_1$, then put $g^* = g^{-1}$.
The corresponding binary 
operation on $S$ is defined as follows. Take $g,h \in S$.
If $(g,h) \in G_2$, then let $gh$ denote the ordinary multiplication in $G_1$.
If $(g,h) \notin G_2$, then put $gh = o$.
If we put $R_o = \{ 0 \}$, then it is clear that $R$ is graded by $S$.
The following result is immediate.

\begin{prop}\label{translation}
Using the above notations, we get:
\begin{itemize}

\item[(a)] $E(S) = G_0 \cup \{ o \}$;

\item[(b)] $R_0 = \oplus_{e \in G_0} R_e$;

\item[(c)] $Z(R_0) = \oplus_{e \in G_0} Z(R_e)$;

\item[(d)] $C_R( Z(R_0) ) = \cap_{e \in G_0} C_R( Z(R_e) )$;

\item[(e)] $R$ is idempotent coherent;

\item[(f)] $R$ is $S$-graded simple if and only $R$ is $G$-graded simple;

\item[(g)] $R$ is left (right) non-degenerate as an $S$-graded ring if and only if
$R$ is left (right) non-degenerate as a $G$-graded ring.

\end{itemize}
\end{prop}

\begin{prop}[Lundstr\"{o}m and \"{O}inert \cite{lundstrom2012}]
If the grading on $R$ is left (right) non-degenerate, then
$R / \cap_{e \in G_0} C_R( Z(R_e) )$ has the ideal intersection property. 
\end{prop}

\begin{proof}
This follows from Proposition \ref{minimalprop} and Proposition \ref{translation}.
\end{proof}

\begin{thm}\label{oinertepsilon}
If $R$ is graded simple, the grading on $R$ is non-degenerate and
$\cap_{e \in G_0} C_R( Z(R_e) ) \subseteq R_0$,
then $R$ is simple. 
\end{thm}

\begin{proof}
This follows from Theorem \ref{newmaintheorem} and Proposition \ref{translation}.
\end{proof}

\begin{cor}\label{oinertepsilonagain}
If $R$ is s-unital epsilon-strongly groupoid graded and $R_0$ is a maximally commutative
subring of $R$, then $R$ is simple if and only if $R$ is graded simple.
\end{cor}

\begin{proof}
This follows from Remark \ref{simpleremark}, Proposition \ref{translation} and Theorem \ref{oinertepsilon}. 
\end{proof}

\begin{rem}
Corollary \ref{oinertepsilonagain} generalizes \cite[Proposition 29]{NOP2016}
from the group graded case to the groupoid graded s-unital case.
\end{rem}

We will now specialize our previous results to
partial groupoid actions on rings.

\begin{defi}
For the rest of the section, let $A$ be a ring.
Recall from \cite{bagio2012} that a {\it partial action} of $G$ on $A$ is a collection
$\alpha = ( A_g , \alpha_g )_{g \in G_1}$, where for each $g \in G_1$,
$A_g$ is an ideal of $A_{c(g)}$, $A_{c(g)}$ is an ideal of $A$, and $\alpha_g : A_{g^{-1}} \to A_g$ 
is a ring isomorphism and the following conditions hold
\begin{itemize}

\item $A = \sum_{e \in G_0} A_e$;

\item $\alpha_e = {\rm id}_{A_e}$, for $e \in G_0$;

\item $\alpha_h^{-1}( A_{g^{-1}} \cap A_h ) \subseteq A_{(gh)^{-1}}$,
for $g,h \in G_2$;

\item $\alpha_g ( \alpha_h (x) ) = \alpha_{gh}(x)$,
for $(g,h) \in G_2$ and $x \in \alpha_h^{-1}(A_{g^{-1}} \cap A_h)$.

\end{itemize} 
The partial action $\alpha$ is called {\it global} if
$\alpha_g \alpha_h = \alpha_{gh}$, for $(g,h) \in G_2$.
In that case, $A_g = A_{c(g)}$, for $g \in G_1$.
We say that $\alpha$ is {\it s-unital} if for every
$g \in G_1$, the ring $A_g$ is s-unital.
An ideal $J$ of $A$ is said to be {\it $G$-invariant} if
for all $g \in G_1$ the inclusion 
$\alpha_g ( J \cap A_{g^{-1}} ) \subseteq J$ holds.
The ring $A$ is called {\it $G$-simple} if $\{ 0 \}$ and $A$ are the 
only $G$-invariant ideals of $A$.
To each partial action $\alpha$ on $A$ one can define the associated
{\it partial skew groupoid ring} $A *_{\alpha} G$
in the following way. As an additive group $A *_{\alpha} G$
is defined as $\oplus_{g \in G_1} A_g \delta_g$,
for some formal symbols $\delta_g$.
The multiplication is defined by the relations
$(a_g \delta_g) (b_h \delta_h) = 
\alpha_g( \alpha_{g^{-1}}(a_g) b_h) \delta_{gh}$,
if $(g,h) \in G_2$, and
$(a_g \delta_g) (b_h \delta_h) = 0$, otherwise,
for $g,h \in G_1$, $a_g \in A_g$ and $b_h \in A_h$.
It is clear that $\alpha$ defines an induced a partial action $\pi$
of $S$ on $A$ and that the corresponding skew inverse semigroup
ring $A \rtimes_{\pi} S$ coincides with $A *_{\alpha} G$. 
If $\alpha$ is global, then the multiplication
simplifies to 
$(a_g \delta_g) (b_h \delta_h) = 
a_g \alpha_g( b_h) \delta_{gh}$, if $(g,h) \in G_2$,
and $A *_{\alpha} G$ is called a {\it skew groupoid ring}.
In that case, if all $A_e$, for $e \in G_0$, coincide
with a ring $B$,
and all $\alpha_g = {\rm id}_B$, for $g \in G_1$,
then $A *_{\alpha} G$ equals the groupoid ring $B[G]$. 
\end{defi}

\begin{thm}\label{oinertgen}
Suppose that $\alpha$ is an s-unital partial action of 
a groupoid $G$ on an $s$-unital ring $A$.
If $A$ is $G$-simple and $C_{A *_{\alpha} G}( Z(A) ) \subseteq A$, 
then $A *_{\alpha} G$ is simple.
If $A$ is commutative, then $A *_{\alpha} G$ is simple if and only if
$A$ is $G$-simple and $A$ is a maximal 
commutative subring of $A *_{\alpha} G$.
\end{thm}

\begin{proof}
This follows from Theorem \ref{newsecondmaintheorem} and Proposition \ref{translation}.
\end{proof}

\begin{rem}
The second part of Theorem \ref{oinertgen} generalizes \cite[Theorem 2.3]{goncalvesoinert2014}
where the corresponding result for groups $G$ was obtained. 
\end{rem}

\begin{exa}
We will now apply Theorem \ref{oinertgen} to a concrete situation
where we have a global groupoid action.
It seems to the author of the present article that this construction
was first introduced in \cite{lundstrom2005}.
Namely, let $L/K$ be a finite separable (not necessarily normal) 
field extension. Let $N$ denote a normal closure of $L/K$ and let 
Gal denote the Galois group of $N/K$. 
Furthermore, let $L_1,\ldots,L_n$ denote the different conjugate
fields of $L$ under the action of Gal. 
If $1 \leq i,j \leq n$, then let $G_{ij}$ denote the set of field
isomorphisms from $L_j$ to $L_i$. 
Let $G$ denote the groupoid defined in the following way.
Put $G_0 = \{ 1,\ldots,n \}$ and
let $G_1$ be the union of the $G_{ij}$, for $1 \leq i,j \leq n$.
If $g \in G_{ij}$, then we put $d(g) = j$ and $c(g) = i$.
Define $A$ to be $L_1 \times \cdots \times L_n$ and put
$e_i = (0,\ldots,0,1,0,\ldots,0)$, with 1 in the $i$th position,
for $i=1,\ldots,n$.
Take $g \in G_{ij}$. Put $A_g = A e_i$.
Define $\alpha_g : A e_j \to A e_i$ in the following way.
If $x \in A e_j$, then $x = (0,\ldots,0,y,0,\ldots,0)$,
for some $y \in L_j$ in the $j$th position.
Put $\alpha_g(x) = (0,\ldots,0,g(y),0,\ldots,0)$,
where $g(y)$ is in the $i$th position.
It is clear that $\alpha$ defines a global groupoid
action of $G$ on $A$.
The corresponding skew groupoid ring $A *_{\alpha} G$
will, from now on, be denoted by $L * G$.
Note that if $L/K$ is a {\it Galois} field extension,
then $G = {\rm Gal}$ and $L * G$ is the classical skew group ring.
\end{exa}

\begin{thm}\label{simplicityLK}
If $L/K$ is a finite separable field extension,
then the corresponding skew groupoid ring $L * G$ is simple.
\end{thm}

\begin{proof}
We wish to use Theorem \ref{oinertgen}.
First we show that $A$ is $G$-simple.
To this end, suppose that $J$ is a non-zero ideal of $A$.
Since $J = J1 = Je_1 \oplus \cdots \oplus Je_n$
it follows that there is $i \in G_0$
with $J e_i \neq \{ 0 \}$. Since $L_i$ is a field
it is clear that $J e_i = A e_i \ni e_i$.
Take $j \in \{ 1,\ldots,n \}$ and $g \in G_{ji}$.
Since $A$ is $G$-invariant it follows that
$e_j = \alpha_g(e_i) \in J$.
Since $j$ was arbitrarily chosen from $G_0$
it follows that $1 = e_1 + \cdots + e_n \in J$.
Thherefore $J = A$.
Now we show that $A$ is maximally commutative in $L * G$.
First of all, it is clear that $A$ is commutative.
Next, let $X$ be a finite subset of $G_1$.
Suppose that $z = \sum_{g \in X} a_g \delta_g \in C_{L * G}(A)$
for some non-zero $a_g \in L_{c(g)}$.
Since $z e_i = e_i z$ for all $i \in G_0$
it follows that $d(g)=c(g)$ for all $g \in X$.
Take $h \in X$ and put $j = d(h) = c(h)$.
Seeking a contradiction, suppose that
$h$ is not equal to the identity element in
the group $G_{jj}$. Then there is 
$t \in L_i$ such that $h(t) \neq t$.
From the relation $z t \delta_i = t \delta_i z$
we now get the contradiction $h(t) = t$.
Therefore $z \in A$ and we are done.
\end{proof}

\begin{rem}
Note that the simplicity of $L * G$ in Theorem \ref{simplicityLK} 
also follows from \cite[Theorem 4]{lundstrom2005}
where a more general result was obtained, however via different
techniques (separability of ring extensions).
\end{rem}

\section*{Acknowledgement}

The author is indebted to the anonymous referee for many valuable comments on the manuscript.

\end{document}